\documentclass[12pt]{amsart}

\usepackage{graphicx}
\usepackage{amsmath}
\usepackage{amscd}
\usepackage{amsfonts}
\usepackage{amssymb}
\usepackage{color}
\usepackage{fullpage}
\usepackage{mathtools}
\usepackage[pagebackref,hypertexnames=false, colorlinks, citecolor=red,linkcolor=blue, urlcolor=red]{hyperref}
\usepackage{todonotes}
\usepackage{enumitem}

\newcommand{\itref}{\ref}

\numberwithin{equation}{section}

\newtheorem{theorem}{Theorem}[section]
\newtheorem{lemma}[theorem]{Lemma}
\newtheorem{proposition}[theorem]{Proposition}

\newtheorem{corollary}[theorem]{Corollary}

\theoremstyle{definition}
\newtheorem{example}[theorem]{Example}
\newtheorem{remark}[theorem]{Remark}
\newtheorem{definition}[theorem]{Definition}
\newtheorem{problem}[theorem]{Problem}

\newcommand{\be}{\begin{equation}}
\newcommand{\ee}{\end{equation}}
\newcommand{\bes}{\begin{equation*}}
\newcommand{\ees}{\end{equation*}}

\newcommand{\cA}{\mathcal{A}}

\newcommand{\cH}{\mathcal{H}}

\newcommand{\cK}{\mathcal{K}}
\newcommand{\cL}{\mathcal{L}}

\newcommand{\bC}{\mathbb{C}}
\newcommand{\bD}{\mathbb{D}}

\newcommand{\bN}{\mathbb{N}}

\newcommand{\bR}{\mathbb{R}}
\newcommand{\bT}{\mathbb{T}}

\newcommand{\ol}{\overline}

\newcommand{\re}{\operatorname{Re}}
\newcommand{\im}{\operatorname{Im}}

\newcommand{\Aut}{\operatorname{Aut}}

\newcommand{\id}{\operatorname{id}}

\newcommand{\spn}{\operatorname{span}}

\begin{document}

\title{Dilations of $q$-commuting unitaries}

\author{Malte Gerhold}
\address{M.G., Institut f\"ur Mathematik und Informatik \\
Ernst Moritz Arndt Universit\"at
Greifswald\\Walther-Rathenau-Stra\ss{}e 47 \\
17487 Greifswald \\ Germany\\}
\address{Faculty of Mathematics\\
Technion - Israel Institute of Technology\\
Haifa\; 3200003\\
Israel}
\email{mgerhold@uni-greifswald.de}
\urladdr{www.math-inf.uni-greifswald.de/index.php/mitarbeiter/282-malte-gerhold}

\author{Orr Moshe Shalit}
\address{O.S.\\Faculty of Mathematics\\
Technion - Israel Institute of Technology\\
Haifa\; 3200003\\
Israel}
\email{oshalit@technion.ac.il}
\urladdr{https://oshalit.net.technion.ac.il/}

\thanks{The work of M. Gerhold is partially supported by the DFG, project no.\  397960675.}
\thanks{The work of O.M. Shalit is partially supported by ISF Grant no.\ 195/16.
}
\subjclass[2010]{47A13, 46L07}
\keywords{Rotation algebra, dilation, commuting unitary dilation, almost Mathieu operator}
\begin{abstract}
Let $q = e^{i \theta} \in \bT$ (where $\theta \in \bR$), and let $u,v$ be $q$-commuting unitaries, i.e., $u$ and $v$ are unitaries such that $vu = quv$.
In this paper we find the optimal constant $c = c_\theta$ such that $u,v$ can be dilated to a pair of operators $c U, c V$, where $U$ and $V$ are commuting unitaries.
We show that
\[
c_\theta = \frac{4}{\|u_\theta+u_\theta^*+v_\theta+v_\theta^*\|},
\]
where $u_\theta, v_\theta$ are the universal $q$-commuting pair of unitaries, and we give numerical estimates for the above quantity.
In the course of our proof, we also consider dilating $q$-commuting unitaries to scalar multiples of $q'$-commuting unitaries.
The techniques that we develop allow us to give new and simple ``dilation theoretic'' proofs of well known results regarding the continuity of the field of rotations algebras.
In particular, for the so-called ``almost Mathieu operator'' $h_\theta = u_\theta+u_\theta^*+v_\theta+v_\theta^*$, we recover the fact that the norm $\|h_\theta\|$ is a Lipschitz continuous function of $\theta$, as well as the result that the spectrum $\sigma(h_\theta)$ is a $\frac{1}{2}$-H{\"o}lder continuous function in $\theta$ with respect to the Hausdorff metric.
In fact, we obtain this H{\"o}lder continuity of the spectrum for every selfadjoint $*$-polynomial $p(u_\theta,v_\theta)$, which in turn endows the rotation algebras with the natural structure of a continuous field of C*-algebras.
\end{abstract}

\maketitle

\section{Introduction}

Let $A_1, \ldots, A_d$ be contractions on a Hilbert space $\cH$ (by \emph{contraction} we mean a linear operator of norm less than or equal to $1$).
It is known \cite{DDSS17,HKMS19} that there exists a constant $c \geq 1$, and commuting normal contractions $B_1, \ldots, B_d$ on a Hilbert space $\cK \supseteq \cH$, such that $A$ is {\em the compression of $cB$ to $\cH$}:
\be\label{eq:dilationAB}
A = c P_\cH B \big|_\cH. 
\ee
By \eqref{eq:dilationAB} we mean that $A_i = c P_\cH B_i \big|_\cH$ for all $i=1, \ldots, d$, where $P_\cH$ denotes the orthogonal projection $P_\cH \colon \cK \to \cH$.
In this case we say that the {\em normal tuple} $cB = (cB_1, \ldots, cB_d)$ is a {\em dilation} of the tuple $A = (A_1, \ldots, A_d)$, and we write $A \prec cB$.
The research behind this paper was motivated by the following problem.

\begin{problem}\label{prob:dilconst}
Fix $d \in \bN$.
What is the smallest constant $C_d$ such that for every $d$-tuple of contractions $A$, there exists a $d$-tuple of commuting normal contractions $B$, such that \eqref{eq:dilationAB} holds with constant $c = C_d$?
\end{problem}

Dilation problems such as this arose in the setting of relaxation of spectrahedral inclusion problems \cite{HKMS19}, in interpolation problems for completely positive maps and the study of the structure of operator systems \cite{DDSS17,FNT17}, and have turned out to be connected to quantum information theory as well \cite{BN18}.
Passer, Shalit and Solel showed that if $A$ is a $d$-tuple of selfadjoint contractions, then there exists a $d$-tuple of commuting selfadjoint contractions $B$ such that \eqref{eq:dilationAB} holds with $c = \sqrt{d}$, and that this is the optimal constant for selfadjoint tuples \cite[Theorem 6.6]{PSS18}.
Moreover, it was shown by Passer in \cite[Theorem 4.4]{Passer} that if $A$ is not assumed selfadjoint, one can do with $c = \sqrt{2d}$.
Thus, we have the bounds
\be\label{eq:Cd}
\sqrt{d} \leq C_d \leq \sqrt{2d}.
\ee

In this paper we make some progress in our understanding of the general constant $C_d$ by studying a certain refinement of Problem \ref{prob:dilconst} which appears to us to be of independent interest.
Instead of dilating arbitrary tuples of contractions, we shall concentrate on dilating pairs of unitaries $u,v$ that satisfy the $q$-commutation relation $vu = quv$, and study the dependence of the dilation on the parameter $q$.
In the context of Problem \ref{prob:dilconst}, it is worth noting that Keshari and Mallick proved that every pair of $q$-commuting contractions has a $q$-commuting unitary (power) dilation \cite{KM19}; our work therefore has immediate implications to all pairs of $q$-commuting operators.

The primary goal here is to find the optimal constant $c_\theta$, such that every pair of $q$-commuting unitaries can be dilated to a pair $c_\theta U,c_\theta V$ where $U$ and $V$ are two commuting unitaries.
Let $u_\theta, v_\theta$ denote the universal $q$-commuting unitaries, where $q = e^{i\theta}$.
Using the fact that the Weyl unitaries can faithfully represent the rotation algebras (i.e., the universal C*-algebras generated by $q$-commuting unitaries), we show that $u_\theta, v_\theta$ is a compression of $c u_{\theta'}, cv_{\theta'}$, where $c = e^{\frac{1}{4}|\theta - \theta'|}$ (Theorem \ref{thm:qqtag}).
This has the important consequence that both the norm $\|u_\theta+u_\theta^*+v_\theta+v_\theta^*\|$ and the dilation constant $c_\theta$ are Lipschitz continuous functions of $\theta$; see Corollaries \ref{cor:Lipnorm} and \ref{cor:ctheta}, respectively.

Our main result is Theorem \ref{thm:opt_dilconst}, which says that
\be\label{eq:ctheta}
c_\theta = \frac{4}{\|u_\theta+u^*_\theta+v_\theta+v^*_\theta\|}.
\ee
This theorem follows from Proposition \ref{prop:lowerbnd}, in which it is shown that the right hand side of the above equation is a lower bound, and Theorem \ref{thm:optdil_qqtag}, a special case of which implies that it is an upper bound.
Theorem \ref{thm:optdil_qqtag} also implies that $u_\theta, v_\theta$ is a compression of $c u_{\theta'}, c v_{\theta'}$ for $c =  c_\gamma$ with $\gamma = \theta' - \theta$, and $c_\gamma$ turns out to be the optimal constant for this kind of dilation (see Theorem \ref{thm:optdil_gen}). 

Our proof in fact shows that every $q$-commuting pair of unitaries $U,V \in B(\cH)$ dilates to a pair $c_\gamma U', c_\gamma V'$ where $U', V'$ are $q'$-commuting unitaries acting on a space $\cK = \cH \otimes\cL$, and that if $q'/q$ is a primitive $n$th root of unity, then we can take $\dim \cL = n$. 
This result should be compared with \cite[Theorem 7.1]{DDSS17} (see also its precursors \cite{MS13} and \cite{Coh15}), which says that whenever a $d$-tuple $X$ of $m \times m$ matrices has any commuting normal dilation $N \in B(\cK)$ with joint spectrum $\sigma(N) \subseteq K$ (for some $K \subseteq \bC^d$), then such a dilation can be chosen so that $\dim \cK \leq 2m^3(d+1) + 1$. 

The operator $h_\theta = u_\theta + u_\theta^* + v_\theta + v_\theta^*$ which appears in the formula \eqref{eq:ctheta} for $c_\theta$ is a so-called ``almost Mathieu operator'', and has been the subject of intensive investigations by mathematicians and physicists. 
This operator is the Hamiltonian of a certain quantum mechanical system consisting of an electron in a magnetic field
\cite{Hof76}. 
If one arranges the spectra of all $h_\theta$ in the planar set $\{(\lambda, \theta) : \theta \in [0,2\pi], \lambda \in \sigma(h_\theta)\}$, then one gets the renowned Hofstadter butterfly \cite[Fig.~1]{Hof76}. 

Corollary \ref{cor:Lipnorm} on the Lipschitz continuity of $\|h_\theta\|$ is not new, and follows, e.g., from results of Bellisard \cite[Proposition 7.1]{Bel94}.
In fact much subtler and sharper facts are known.
For example, in \cite{CEY} Choi, Elliott, and Yui showed that the spectrum $\sigma(h_\theta)$ of $h_\theta$ depends H{\"o}lder continuously (in the Hausdorff metric) on $\theta$, with H{\"o}lder exponent $1/3$.
This was soon improved by Avron, Mouche, and Simon to H{\"o}lder continuity with exponent $1/2$ \cite{AMS90}.
In Section \ref{sec:cont-field}, we shall recover known generalizations of these results by what we consider to be much simpler proofs (see Theorem \ref{thm:Holder}).
(It is no coincidence that $\|h_\theta\|$ is Lipschitz continuous and $\sigma(h_\theta)$ is 1/2-H{\"o}lder continuous --- there are very general results connecting continuity of the norms and of the spectra; see \cite{BB16}.
The $1/2$-H{\"o}lder continuity of the spectrum also follows from the result of Haagerup and R{\o}rdam, who showed that there exist $1/2$-H{\"o}lder {\em norm} continuous paths $\theta \mapsto u_\theta \in B(\cH)$, $\theta \mapsto v_\theta \in B(\cH)$ \cite[Corollary 5.5]{HR95}.)
Consequently, we easily obtain the familiar fact that the rotation algebras form a continuous field of C*-algebras \cite{Ell82}.
For the sake of completeness, and since we think that the simple proofs we provide are of interest, we give the details (see Corollary \ref{cor:Hcontfield}).

Finally, we come back to Problem \ref{prob:dilconst}. 
We were somewhat surprised that, notwithstanding the enormous amount of work on the almost Mathieu operators $h_\theta$, it seems that it is still not known precisely what is the minimal value of the norm $\|h_\theta\|$ or the angle $\theta$ where this minimum is attained. The only result we know of appears in \cite{BZ05} (in partiuclar see Equation (1.13) there), where it is proved that $\|h_\theta\|\geq 2.56769$ for all $\theta$, leading to $\max_\theta c_\theta\leq \frac{4}{2.56769}\leq 1.558$.
In Section \ref{sec:numerical} we obtain numerical values for $c_\theta = 4/\|h_\theta\|$ for various $\theta$.
In Example \ref{ex:45} we calculate by hand $c_{\frac{4}{5}\pi} \approx 1.5279$,
and this allows us to push the lower bound $C_2 \geq 1.41...$ to $C_2 \geq 1.52$.
We also report on modest numerical computations that we carried out, which lead to an improved estimate $C_2 \geq \max_\theta c_\theta \geq 1.5437$ (and, therefore, $\min_\theta\|h_\theta\|\leq \frac{4}{1.5437}\leq 2.5912$, which is in agreement with the numerical results mentioned in \cite{BZ05}). 
The latter value is an approximation of the constant $c_{\theta_s}$ attained at the {\em silver mean} $\theta_s = 2\pi(\sqrt{2}-1)$, which we conjecture to be the angle where the maximum is attained.
However, we do not expect that the maximal value of $c_\theta$ will give a tight lower approximation for $C_2$ --- for obtaining the value of $C_2$ we will probably have to study dilations of free unitaries in more detail.

\begin{remark}
In the theory of operator spaces, there are the notions of minimal and maximal operator spaces over a normed space $V$, and there is a constant $\alpha(V)$ that quantifies the difference between the minimal and maximal operator space structures \cite{Pau92} (see also \cite[Chapter 14]{PauBook} and \cite[Chapter 3]{PisBook}). 
Additionally, there are notions of minimal and maximal operator systems \cite{PTT11}, closely related to the notions of minimal and maximal matrix convex sets \cite{DDSS17,FNT17,HKM16}. 
Given a convex set $K \subseteq \bC^d$, there is a constant $\theta(K)$ that quantifies the difference between the minimal and maximal matrix convex sets over $K$ \cite[Section 3]{PSS18}. 
The constant $C_d$ from Problem \ref{prob:dilconst} is neither $\alpha(\ell^\infty_d)$ nor is it $\theta(\ol{\bD}^d)$, as one might mistakenly think. 
The place of Problem \ref{prob:dilconst} within the theory of operator spaces has yet to be determined. 
\end{remark}

In September 2018, Mattya Ben-Efraim and Yuval Yifrach participated in the program ``Summer Research Projects'' in the Department of Mathematics at the Technion, and worked on a project under the supervision of the second author.
They ran numerical experiments on random matrices and found several pairs of $4 \times 4$ unitaries which cannot be dilated to commuting unitaries times $\sqrt{2}$, thereby showing that $C_2 > \sqrt{2}$.
Interestingly, they also observed that such ``counter examples'' seem to be rare, and to enjoy a certain structure.
Subsequently, following discussions with Benjamin Passer, they numerically computed $c_{\frac{2\pi}{3}}$ (using an adaptation of the algorithm from \cite{HKM13}, implemented in MATLAB with the package CVX) and found that it is also slightly bigger than $\sqrt{2}$ (see Example~\ref{ex:23} for the exact value).
It is important for us to point out that the research behind this paper was stimulated by their work.

\section{Some facts about rotation C*-algebras}

\begin{definition}\label{def:rotation-algebra}
Let $\theta\in \mathbb R$ and write $q=e^{i\theta}$.
We denote by $A_\theta$ the universal C*-algebra
\[
A_\theta=C^*(u,v \text{ unitary}\mid vu=quv),
\]
and we call $A_\theta$ a \emph{rational/irrational rotation C*-algebra} if $\frac{\theta}{2\pi}$ is rational/irrational respectively.
We shall write $u_\theta, v_\theta$ for the generators of $A_\theta$.
\end{definition}

Ever since they were introduced by Rieffel \cite{Rie81}, the rotation algebras have been of widespread interest to operator algebraists and mathematical physicists alike.
We will make use of the following well known facts about rotation C*-algebras; see, e.g., Boca's excellent book on the subject \cite[Theorems 1.9 and 1.10]{BocBook}.

\begin{itemize}
\item The irrational rotation C*-algebras are simple. 
In particular, for every pair of $q$-commuting unitaries $U,V\in B(\cH)$, the corresponding representation of $A_\theta$ is isometric, i.e., it yields an isomorphism $A_\theta\cong C^*(U,V)$; see also \cite[Chapter VI]{DavBook}.
\item Let $U,V$ be $q$-commuting unitaries with $n:=\min\{k\in\mathbb N\mid q^k=1\}<\infty$. 
Put
\be\label{eq:XY}
X=\mathrm{diag}(1,q,\ldots,q^{n-1}) \quad ,\quad Y = \begin{pmatrix}  & 1 &  &  &  \\  &  & 1 & & \\ & & & \ddots & \\ & & & & 1 \\ 1 & & & & \end{pmatrix}\quad .
\ee
Then every irreducible representation of $C^*(U,V)$ has dimension $n$ and is of the form
\begin{align*}
\pi(U)=\alpha X \, , \, \pi(V)=\beta Y
\end{align*}
for some $\alpha,\beta\in\mathbb T$.
\item Of course, if $C^*(u_\theta,v_\theta)$ is a rational rotation $C^*$-algebra, then there is an irreducible representation $\pi_{\alpha,\beta}$ with $\pi_{\alpha,\beta}(u)=\alpha X$ and $\pi_{\alpha,\beta}(v)=\beta Y$ for all $\alpha,\beta\in\mathbb T$.
Two such representations $\pi_{\alpha,\beta}, \pi_{\alpha',\beta'}$ are unitarily equivalent if and only if $\alpha'=q^k \alpha$ and $\beta'=q^\ell\beta$ for some $k,\ell\in\mathbb N$. We refer to $\pi_{1,1}$ as the \emph{standard representation} of $A_\theta$.
\end{itemize}

We shall consider $*$-polynomials in two noncommuting variables.
A \emph{$*$-polynomial} $p = p(x_1, x_2)$ in two noncommuting variables is nothing but a polynomial $q = q(z_1, z_2, z_3, z_4)$ (with complex coefficients) in four noncommuting variables evaluated at the tuple of variables $(x_1,x_1^*,x_2,x_2^*)$, that is $p(x_1,x_2) = q(x_1,x_1^*,x_2,x_2^*)$.
A $*$-polynomial $p(x_1,x_2)$ is said to be \emph{selfadjoint} if the formal application of the adjoint leaves it invariant; this implies that $p(a,b)$ is a selfadjoint operator for every pair of operators $a,b$.

\begin{lemma}\label{RR}
Let $A_\theta$ be a rational rotation C*-algebra and $U,V\in B(\cH)$ a pair of $q$-commuting unitaries, $q=e^{i\theta}$. 
Then the following are equivalent.
\begin{enumerate}[label=\textnormal{(\arabic*)}]
\item\label{RR:it:1} There is an isomorphism $A_\theta\cong C^*(U,V)$ with $u_\theta\mapsto U$ and $v_\theta\mapsto V$.
\item\label{RR:it:2} For all $\alpha,\beta\in\mathbb{T}$ there exists a $*$-automorphism $\Phi_{\alpha,\beta}$ of $C^*(U,V)$ with $\Phi_{\alpha,\beta}(U)=\alpha U$ and $\Phi_{\alpha,\beta}(V)=\beta V$.
\item\label{RR:it:3} For all $\alpha,\beta\in\mathbb T$ there is a representation of $C^*(U,V)$ with $U\mapsto\alpha X$ and $V\mapsto\beta Y$.
\end{enumerate}
\end{lemma}

\begin{proof}
$\itref{RR:it:1}\implies\itref{RR:it:2}$ follows easily from the universal property of $A_\theta$.

For $\itref{RR:it:2}\implies\itref{RR:it:3}$, choose an irreducible representation $\pi_0$ of $C^*(U,V)$.
Without loss of generality, we assume that $\pi_0(U)=\alpha_0X$ and $\pi_0(V)=\beta_0Y$ for some $\alpha_0,\beta_0\in\mathbb T$.
We get a representation $\pi$ of $C^*(U,V)$ with $\pi(U)=\alpha X,\pi(V)=\beta Y$ as $\pi:=\pi_0\circ\Phi_{\alpha\overline{\alpha_0},\beta\overline{\beta_0}}$.

Now we prove $\itref{RR:it:3}\implies\itref{RR:it:1}$.
For every $*$-polynomial $P$, we get
\begin{align*}
\|P(U,V)\|
\geq \sup_{\pi\in\mathrm{Rep}(C^*(U,V))} \|P(\pi(U),\pi(V))\|
\geq \sup_{(\alpha, \beta) \in \bT^2} \|P(\alpha X,\beta Y)\|\\
= \sup_{\pi\in\mathrm{Irr}(A_\theta)}\|P(\pi(u_\theta),\pi(v_\theta))\|=\|P(u_\theta,v_\theta)\|,
\end{align*}
so the surjective $*$-homomorphism from $A_\theta$ to $C^*(U,V)$ determined by $u_\theta\mapsto U, v_\theta\mapsto V$ has to be isometric and, thus, an isomorphism.
\end{proof}

Note that the condition that $U$ and $V$ have full spectrum, i.e., $\sigma(U)=\sigma(V)=\mathbb T$, is not enough to guarantee that $(U,V)$ generate a rotation algebra. Indeed, let $w$ be a unitary with full spectrum and $X,Y$ be the $q$-commuting matrices as in \eqref{eq:XY}. Then $U:=X\otimes w$ and $V:=Y\otimes w$ are $q$-commuting unitaries with full spectrum, but $\sigma(UV^*)$ is a finite set whereas $\sigma(u_\theta v_\theta^*) = \bT$.

We will require some basic facts on completely positive maps (see \cite{PauBook}).
If $A$ and $B$ are unital C*-algebras, a linear map $\phi \colon A \to B$ is said to be {\em positive} if $\phi(a) \geq 0$ whenever $a\geq 0$ and {\em unital} if $\phi(1) = 1$. 
A linear map $\phi \colon A \to B$ promotes to a map $\phi_n \colon M_n(A) \to M_n(B)$ by acting componentwise, and is said to be {\em completely positive} if $\phi_n$ is positive for all $n$.
A unital and completely positive map will be called a {\em UCP map}. 

By Stinespring's theorem \cite{Sti55}, for every UCP map $\phi \colon A \to B(\cH)$, there exists a Hilbert space $\cK$ containing $\cH$, and a unital $*$-representation $\pi \colon A \to B(\cK)$, such that
\[
\phi(a) = P_\cH \pi(a) \big|_\cH  
\]
for all $a \in A$, where $P_\cH$ is the orthogonal projection of $\cK$ onto $\cH$. 

An {\em operator system} is a vector subspace $S$ of a unital C*-algebra $A$ such that $1 \in S$ and $S = S^*$. 
An operator system $S$ is spanned by its positive elements, and the definitions of positive, completely positive, and UCP maps make sense for linear maps between operator systems. 
Arveson's extension theorem \cite[Theorem 1.2.3]{Arv69} states that if $S$ is an operator system contained in a unital C*-algebra $A$, then every UCP map of $S$ into $B(\cH)$ extends to a UCP map from $A$ into $B(\cH)$.

It is known that, roughly speaking, the existence of dilations is equivalent to the existence of UCP maps, as well as to certain matrix valued linear inequalities. 
We formulate this observation in a form immediately usable for our purposes. 

\begin{proposition}\label{prop:equiv}
Let $\theta,\theta'\in \bR$, put $q=e^{i\theta}$, $q' = e^{i\theta'}$, and let $c > 0$ be a constant.
The following are equivalent.
\begin{enumerate}[label=\textnormal{(\arabic*)}]
\item\label{equiv:it:1} Every $q$-commuting pair of unitaries $U,V$ can be dilated to a pair $cU',cV'$, where $U',V'$ are unitaries that $q'$-commute.
\item\label{equiv:it:2} The universal $q$-commuting pair of unitaries $u_\theta,v_\theta$ can be dilated to a pair $cU',cV'$, where $U',V'$ are unitaries that $q'$-commute.
\end{enumerate}
If $q' = 1$, then the above conditions are equivalent to the condition:
\begin{enumerate}[label=\textnormal{(\arabic*)},resume]
\item\label{equiv:it:3} Every $q$-commuting pair of unitaries $U,V$ can be dilated to a pair of commuting normal operators $M, N$ such that $\|M\|, \|N\| \leq c$.
\end{enumerate}
\end{proposition}

\begin{proof}
For general $q'$, we need only prove $\itref{equiv:it:2} \implies \itref{equiv:it:1}$.
Let $U,V$ be a $q$-commuting pair, and let $u_\theta,v_\theta$ be a universal $q$-commuting pair of unitaries generating $A_\theta$.
There is a surjective $*$-homomorphism $\pi \colon A_\theta \to C^*(U,V)$ mapping $u_\theta \mapsto U$ and $v_\theta \mapsto V$. 
Assume that $A_\theta$ is represented faithfully on a Hilbert space $\cH_\theta$, and that $C^*(U,V)$ is represented faithfully on a Hilbert space $\cH$; in other words, suppose $A_\theta\subset B(\cH_\theta)$ and $C^*(U,V)\subset B(\cH)$. 
By Arveson's extension theorem, $\pi$ extends to a UCP map $\widetilde\pi\colon B(\cH_\theta)\to B(\cH)$.  

By assumption, there is a $q'$ commuting pair $U',V'$ such that $cU',cV'$ is a dilation of $u_\theta,v_\theta$.
Composing the compression and $\widetilde{\pi}$, we obtain a UCP map $\phi \colon C^*(U',V') \to B(\cH)$ such that $\phi(cU') = U$ and $\phi(cV') = V$.
If $\sigma$ is the Stinespring dilation of $\phi$, then $c\sigma(U'),c\sigma(V')$ is the sought after $q'$-commuting dilation of $U,V$.

For the case $q' = 1$, we note that if $U,V$ can be dilated to a pair of commuting normal operators $M,N$ with joint spectrum in a compact convex set $L \subset \bC^2$, then $U,V$ can also be dilated to a pair of commuting normal operators $\tilde{M}, \tilde{N}$ with joint spectrum contained in the closure $\overline{\operatorname{ext}(L)}$ of the set of extreme points of $L$ (the proof of this fact follows ideas similar to the ones in the previous paragraphs; see \cite[Proposition 2.3]{PSS18} for details).
It follows that the existence of a commuting normal dilation implies the existence of a dilation by an appropriate scalar multiple of commuting unitaries.

Alternatively, again for the case $q' = 1$, we note (using Fuglede's and Stone-Weierstrass' theorems, for example) that if $M,N$ are commuting normal contractions, then
\[
U' = \begin{pmatrix} M & (1 - M M^*)^{1/2} \\ (1 - M^* M)^{1/2} & -M^* \end{pmatrix} \quad , \quad V' = \begin{pmatrix} N & 0 \\ 0 & N \end{pmatrix}
\]
are a commuting pair of normals that dilate $M$ and $N$, and $U'$ is a unitary.
Repeating this procedure with the roles reversed, we find that if the pair $U,V$ can be dilated to commuting normal contractions, then $U,V$ can be dilated to commuting unitaries.
From this one easily sees that \itref{equiv:it:3} implies \itref{equiv:it:2}.
The converse implication is trivial.
\end{proof}

\begin{remark}
We presented the preceding proposition and its proof in a way that fits our narrative and makes the paper essentially self-contained. 
Moreover, the proof presented above generalizes readily to the case of $d$-tuples instead of pairs. 
However, it is worth noting that by relying on Ando's dilation theorem and its generalization to $q$-commuting pairs of contractions \cite{KM19}, one can easily obtain equivalence of \itref{equiv:it:1} and \itref{equiv:it:2} with the condition:
  \begin{itemize}
\item[(3')] Every $q$-commuting pair of unitaries $U,V$ can be dilated to a pair of $q'$-commuting operators $S, T$ such that $\|S\|, \|T\| \leq c$.
\end{itemize}
This stronger conclusion is a special feature of the case $d=2$, and it will have no counterpart in the higher dimensional case. 
\end{remark}

\section{Weyl unitaries and continuity of the dilation scale}\label{sec:weyl-unitaries}

We study a concrete realization of the rotation algebras given by the well known Weyl unitaries.
Our reference for the Weyl unitaries is \cite[Section 20]{Par12}.
Following \cite{Par12}, our inner products will be linear in the {\em second} variable.
For a Hilbert space $H$ let
\[
\Gamma(H):=\bigoplus_{k=0}^{\infty} H^{\otimes_s k}
\]
be the symmetric Fock space over $H$.
The exponential vectors $e(x):=\sum_{k=0}^{\infty}\frac{1}{\sqrt{k!}} x^{\otimes k}, x\in H$ form a linearly independent and total subset of $\Gamma(H)$.
Clearly, $\langle e(x),e(y)\rangle = e^{\langle x,y\rangle}$ for all $x,y\in H$. For $z\in H$ we define the Weyl unitary $W(z)\in B(\Gamma(H))$ which is determined by
\[W(z) e(x)=e(z+x) \exp\left(-\frac{\|z\|^2}{2} - \langle z,x\rangle \right)\]
for all exponential vectors $e(x)$.
A simple calculation shows that $W(z),W(y)$ commute up to the phase factor $e^{2i \im \langle z,y\rangle}$; to be precise:
\[
W(y) W(z) = e^{2i \im \langle z,y\rangle} W(z) W(y).
\]

\begin{proposition}\label{prop:WeylGenAt}
For linearly independent $z,y \in H$, the operators $W(z)$ and $W(y)$ are the generators of a rotation C*-algebra $A_\theta$ for $\theta = 2\im \langle z,y\rangle$.
In other words, the $*$-representation of $A_{\theta}$ determined by $u_\theta\mapsto W(z),v_\theta\mapsto W(y)$ is isometric.
\end{proposition}

\begin{proof}
In the irrational case there is nothing to prove.
In the rational case, we use Lemma \ref{RR}.
Note that for every $x \in H$,
\[W(x)^*W(z)W(x)=W(z)e^{2i\im\langle x,z\rangle},\quad W(x)^*W(y)W(x)=W(y)e^{2i\im\langle x,y\rangle}.
\]
Let $\alpha=e^{is},\beta=e^{it}\in\mathbb T$.
Since $z,y$ are linearly independent, there exists an $x\in \spn\{y,z\}$ such that $\langle x,z\rangle=\frac{is}{2}$ and $\langle x,y\rangle=\frac{it}{2}$, so that $e^{2i\im\langle x,z\rangle}=\alpha$,  $e^{2i\im\langle x,y\rangle}=\beta$.
Clearly, conjugation with the unitary $W(x)$ is an automorphism of $C^*(W(z),W(y))$, so we showed condition \itref{RR:it:2} of Lemma \ref{RR}.
\end{proof}

\begin{theorem}\label{thm:qqtag}
Let $\theta, \theta' \in \mathbb R$, set $q=e^{i\theta}, q'=e^{i\theta'}$, and put $c = e^{\frac{1}{4}|\theta - \theta'|}$.
Then for any pair of
$q$-commuting unitaries $U,V$ there exists a pair of $q'$-commuting unitaries $U',
V'$ such that $cU', cV'$ dilates $U, V$.
\end{theorem}

\begin{proof}
Consider Hilbert spaces $H\subset K$ with $p$ the projection onto $H$,
and the symmetric Fock spaces $\Gamma(H) \subset \Gamma(K)$ with $P$ the
projection onto $\Gamma(H)$.
We write $p^\perp $ for the projection onto the orthogonal complement $H^\perp $.
Note that for exponential vectors we have
$Pe(x)=e(px)$.
For every $y,z\in K$, the Weyl unitaries $W(y), W(z)$ satisfy:
\begin{enumerate}[label=\textnormal{(\arabic*)}]
\item $W(z),W(y)$ commute up to the phase factor $e^{2i \im \langle z,y\rangle}$.
\item $PW(z)\big|_{\Gamma(H)}= e^{-\frac{\|p^\perp  z\|^2}{2}}W(pz)$, so it is a scalar multiple of a unitary on $\Gamma(H)$.
\item $PW(z)\big|_{\Gamma(H)}, PW(y)\big|_{\Gamma(H)}$ commute up to the phase factor $e^{2i \im \langle pz,py\rangle} = e^{2i \im \langle z,py\rangle}$.
\end{enumerate}

Suppose without loss of generality that $0\leq\theta'<\theta\leq2\pi$ (if $\theta<\theta'$, then $2\pi-\theta>2\pi-\theta'$ and we can repeat the following argument for the complex conjugates $\overline q$ and $\overline {q'}$ by flipping $(U,V)$ to $(V,U)$ and $(U',V')$ to $(V',U')$; alternatively we can modify the step in the argument below in which parameters are chosen).
We claim that we can arrange things so that there are two linearly independent vectors $z,y$ so that $pz$ and $py$ are also linearly independent, and such that
\begin{enumerate}[label=\textnormal{(\arabic*)}]
\item $p^\perp y=-ip^\perp z$,
\item $\theta'=2\im \langle z,y\rangle$,
\item $\theta= 2\im \langle z,py\rangle$.
\end{enumerate}
Supposing for the moment that this can be done, then we get $q'$-commutation of
$W(z),W(y)$, $q$-commutation of $PW(z)\big|_{\Gamma(H)}, PW(y)\big|_{\Gamma(H)}$ and
\[
\theta-\theta' = -2\im \langle z,p^\perp  y\rangle = 2\|p^\perp z\|^2 = 2\|p^\perp y\|^2,
\]
so
\[
\left\|PW(z)\big|_{\Gamma(H)}\right\|=\left\|PW(y)\big|_{\Gamma(H)}\right\|=e^{-\frac{\|p^\perp  y\|^2}{2}}=e^{-\frac{|\theta-\theta'|}{4}} .
\]
Now put
\[
U= e^{\frac{|\theta-\theta'|}{4}}PW(z)\big|_{\Gamma(H)} \,\, , \,\, V= e^{\frac{|\theta-\theta'|}{4}} PW(y)\big|_{\Gamma(H)},
\]
and
\[
U'=W(z) \,\, , \,\, V'=W(y)
\]
to get the statement for this particular $q$-commuting pair $U,V$.

To find the $z$ and the $y$ above, let $H$ be a two dimensional space and let $z',y'$ be two linearly independent vectors such that $2 \im \langle z', y' \rangle = \theta$ (for example, one can take $z' = e_1 + e_2$ and $y' = e_1+i\frac{\theta}{2} e_2$, where $e_1,e_2$ are orthogonal unit vectors).
Now let $K$ be a proper superspace of $H$, and let $w \in H^\perp $ be a unit vector.
Define
\[
z = z' + \alpha w \quad \textrm{ and } \quad y = y' - i \alpha w
\]
for some real parameter $\alpha$ to be chosen soon.
Then clearly $p^\perp y = -i p^\perp z$ and $2\im \langle z, py\rangle = 2\im \langle z', y'\rangle = \theta$.
To fulfill all the requirements, it remains to satisfy
\[
2\im \langle z,y\rangle = \theta - 2 \alpha^2 = \theta',
\]
and clearly $\alpha$ can be chosen so that this holds.

So far, we proved the statement in the case that $U,V$ are a $q$-commuting pair represented by the Weyl operators of two appropriate linearly independent vectors.
By Proposition \ref{prop:WeylGenAt}, such $U, V$ constitute a universal pair of generators of $A_\theta$.
To obtain the statement for an arbitrary $q$-commuting pair $U,V$ we invoke Proposition \ref{prop:equiv}.
\end{proof}

Recall that a function $F\colon I \to \bR$ from an interval $I$ into $\bR$ is {\em Lipschitz continuous with contant $C$} if $|F(t) - F(s)| \leq C|t-s|$ for all $s,t \in I$. 
It is a basic fact that if $F$ is {\em locally Lipschitz continuous with constant $C$} in the sense that for every $t \in I$ there is a neighborhood $U$ of $t$ such that $F$ is Lipschitz with constant $C$ in $U$, then $F$ is Lipschitz with constant $C$. 

\begin{corollary}\label{cor:Lipnorm}
The norm of any matrix valued polynomial of degree one in $u_\theta,u_\theta^*,v_\theta,v_\theta^*$, and in particular the norm of the operator $h_\theta = u_\theta + u_\theta^* + v_\theta + v_\theta^*$, depends Lipschitz continuously on $\theta$.
\end{corollary}

\begin{proof}
For every such polynomial $p$, let $M = \sup_\theta \|p(u_\theta, v_\theta)\|$, which is finite by an elementary estimate (we are not assuming yet that $\theta \mapsto \|p(u_\theta, v_\theta)\|$ is continuous). 
By the theorem, 
\[
p(u_\theta,v_\theta) \prec e^{\frac{1}{4}|\theta - \theta'|} p(u_{\theta'},v_{\theta'}) 
\] 
for every $\theta, \theta'$.  
Assume without loss of generality that $\|p(u_\theta,v_\theta)\| \geq \|p(u_{\theta'},v_{\theta'})\|$. Then, 
\begin{align*}
\Big|\|p(u_\theta,v_\theta)\| - \|p(u_{\theta'},v_{\theta'})\| \Big| 
&\leq  e^{\frac{1}{4}|\theta - \theta'|} \|p(u_{\theta'},v_{\theta'})\| - \|p(u_{\theta'},v_{\theta'})\|  \\
& \leq \left(e^{\frac{1}{4}|\theta - \theta'|} - 1\right)M . 
\end{align*}
Now, $e^{\frac{1}{4}|\theta - \theta'|} = 1 + \frac{1}{4}|\theta - \theta'| + $ higher order terms, so for every $\varepsilon > 0$, the function $\theta \mapsto \|p(u_\theta,v_\theta)\|$ is locally Lipschitz continuous with constanst $(\frac{1}{4}+\varepsilon) M$. 
We conclude that this function has Lipschitz constant $\frac{M}{4}$, and, in particular, the function $\theta \mapsto \|h_\theta\|$ is Lipschitz continuous with constant $1$.
\end{proof}

For every $\theta \in \bR$ we write $q = e^{i\theta}$.
We define the optimal dilation scale
\[
c_\theta:=\inf\{c > 1\mid \exists \textrm{ a commuting normal dilation for } u_\theta,v_\theta \text{  with norm }\leq \, c\}
\]
and note that by Proposition~\ref{prop:equiv} this is the same as the infimum of the constants $c$ that satisfy: for every $q$-commuting pair of unitaries $U,V$ there exists a commuting normal dilation $M,N$ such that $\|M\|,\|N\|\leq c$.
In the following sections we determine $c_\theta$, and our proof of Theorem~\ref{thm:optdil_qqtag} will show that the infimum is actually a minimum (that the infimum is actually attained could also be proved by applying general principles, such as the ideas in Proposition \ref{prop:equiv}, together with the compactness of UCP maps in an appropriate topology).
Theorem \ref{thm:qqtag} also has the following interesting corollary.
\begin{corollary}\label{cor:ctheta}
The optimal dilation scale $c_\theta$
depends Lipschitz continuously on $\theta$. More precisely, for all $\theta,\theta'\in\mathbb{R}$ we have
\[\left|c_\theta-c_{\theta'}\right|\leq \frac{\max_\theta c_\theta}{4}\left|\theta-\theta'\right|\leq 0.39\left|\theta-\theta'\right|. \]
\end{corollary}

\begin{proof}
By the comments made before Corollary \ref{cor:Lipnorm}, it is enough to check the Lipschitz condition locally. 
Fix $\varepsilon>0$ and assume that $\theta$ and $\theta'$ are close enough to have $e^{\frac{1}{4}|\theta - \theta'|}\leq 1+\frac{1}{4-\varepsilon} |\theta - \theta'|$. We find 
  \[
    c_{\theta'}\leq c_\theta\cdot e^{\frac{1}{4}|\theta - \theta'|}\leq c_\theta \left(1 + \frac{1}{4-\varepsilon}|\theta - \theta'|\right)\leq c_\theta + \frac{\max_\theta c_\theta}{4-\varepsilon}|\theta - \theta'|. 
  \]
Thus the function $c_\theta$ is locally Lipschitz with constant $\frac{\max_\theta c_\theta}{4-\varepsilon}$, and as we remarked, this is also a global Lipschitz constant. 
As $\varepsilon$ was arbitrary, we get the first inequality. From this, using also the elementary bound $C_2 \leq 2$, we obtain $\frac{\max_\theta c_\theta}{4}\leq \frac{C_2}{4}\leq \frac{1}{2}$. However, in Theorem~\ref{thm:opt_dilconst} we will actually prove that $c_\theta=\frac{4}{\|h_\theta\|}$. Therefore, Equation (1.13) in \cite{BZ05} implies that $\frac{\max_\theta c_\theta}{4} =  \frac{1}{\min_\theta \|h_\theta\|}\leq \frac{1}{2.56769}\leq 0.39$.
\end{proof}

\begin{remark}
In a previous version of this paper we proved that $c_\theta$ is Lipschitz continuous with Lipschitz constant $1$. We are grateful to an anonymous referee who suggested the above line of reasoning that gives the considerably sharper constant above (even without invoking the delicate results in \cite{BZ05}, and relying only on the elementary inequality $C_2\leq2$, the above argument gives a Lipschitz constant $1/2$). 
\end{remark}

\section{Continuity of the rotation algebras via dilations}\label{sec:cont-field}\label{sec:continuity}

\begin{lemma}
Let $a_1,\ldots,a_n,b_1,\ldots, b_n$ be contractions on the same Hilbert space. Then
\[
\|a_1\cdots a_n-b_1\cdots b_n\|\leq \|a_1-b_1\|+\cdots + \|a_n-b_n\|.
\]
\end{lemma}

\begin{proof}
For $n=2$, we find
\[
\left\|a_1 a_2-b_1 b_2\right\|\leq \left\|a_1a_2-b_1a_2\right\|+\left\|b_1a_2-b_1b_2\right\|\leq   \left\|a_1-b_1\right\|+ \left\|a_2-b_2\right\|
\]
and induction on $n$ proves the general statement.
\end{proof}

\begin{corollary}\label{cor:pol-diff}
Let $p=\sum \alpha_{i_1,\ldots i_k}x_{i_1}\cdots x_{i_k}$ be a polynomial in $n$ noncommuting indeterminates and put
\[
\kappa_p:= \sum k\cdot \left|\alpha_{i_1,\ldots i_k}\right|.
\]
Then for all $n$-tuples of contractions $(a_1,\ldots, a_n),(b_1,\ldots, b_n)$ on the same Hilbert space we have
\[
\|p(a_1,\ldots a_n)-p(b_1,\ldots, b_n)\|\leq \kappa_p\max(\|a_1-b_1\|,\ldots,\|a_n-b_n\|).
\]
\end{corollary}

\begin{lemma}\label{lem:dil-approx}
Let $\cH\subset \cK$ be Hilbert spaces, $U\in B(\cK)$ and $U'\in B(\cH)$ unitaries, and $c\in[0,1]$ such that $cU'$ is the compression of $U$ to $\cH$.
Denote by $R$ the compression of $U$ to $\cH^{\perp}$.
Then
\[
\left\|U-(U'\oplus R)\right\|\leq 2\sqrt{1-c^2}
.
\]
\end{lemma}

\begin{proof}
Write
\[
U = \begin{pmatrix}
    cU'&E\\F&R
  \end{pmatrix}.
\]
Then by unitarity of $U$ and $U'$ we find
\[
c^2 + EE^*=P_{\cH}UU^*\big|_{\cH}=\id_\cH= P_\cH U^*U\big|_{\cH}=c^2+F^*F ,
\]
from which we conclude $\left\|E\right\|=\left\|F\right\|=\sqrt{1-c^2}$.
Therefore,
\begin{align*}
\left\|U-(U'\oplus R)\right\|
& = \left\|
        \begin{pmatrix}
          (c-1)U'&E\\F&0
        \end{pmatrix}
      \right\| \\
& \leq
    \left\|
        \begin{pmatrix}
          (c-1)U'&0\\0&0
        \end{pmatrix}
      \right\|
    +
    \left\|
        \begin{pmatrix}
          0&E\\F&0
        \end{pmatrix}
      \right\| \\
& = 1-c+\sqrt{1-c^2}.
  \end{align*}
Since $1-c\leq \sqrt{1-c^2}$ for all $c\in[0,1]$, we are done.
\end{proof}

\begin{theorem}\label{thm:Holder}
Let $p$ be a selfadjoint $*$-polynomial in two noncommuting variables.
Then the spectrum $\sigma(p(u_\theta,v_\theta))$ is $\frac{1}{2}$-H{\"o}lder continuous in $\theta$ with respect to the Hausdorff distance for compact subsets of $\mathbb C$.
\end{theorem}

\begin{proof}
We will require the fact that for selfadjoint operators $a,b$ it holds that $\sigma(b) \subseteq \sigma(a) + \|b-a\|\cdot[-1,1]$ (or, in other words, $d(\sigma(a),\sigma(b))\leq \|a-b\|$). 
This is well known, but we include a short proof to keep this paper self-contained (see, e.g., \cite[Lemma 1.17]{BocBook} for a different argument). 
Let us show $0\notin \sigma(a) + \|b-a\|\cdot[-1,1]$ implies $0\notin\sigma(b)$, which is of course sufficient, as we can replace $a,b$ with $a-\lambda, b-\lambda$.
So suppose $d(0,\sigma(a))>\left\|b-a\right\|$.
Then
\[
\left\|a^{-1}b-1\right\|\leq \left\|a^{-1}\right\|\left\|b-a\right\|=d(0,\sigma(a))^{-1}\left\|b-a\right\|<1,
\]
so $a^{-1}b$ is invertible, hence $b$ is invertible.

By Theorem \ref{thm:qqtag}, we can write $u_\theta, v_\theta$ as a dilation of $cu_{\theta'}, cv_{\theta'}$ with $c = \exp({-\frac{\left\vert\theta-\theta'\right\vert}{4}})$.
Suppose for definiteness that $u_\theta,v_\theta \in B(\cK)$ and $u_{\theta'}, v_{\theta'} \in B(\cH)$, where $\cH$ is a subspace of $\cK$.
Let $r,s$ denote the compression of $u_{\theta},v_{\theta}$ to $\cH^{\perp}$.
Now
\begin{align*}
\sigma(p(u_{\theta'},v_{\theta'}))
& \subseteq \sigma(p(u_{\theta'},v_{\theta'})\oplus p(r,s)) \\
& \subseteq \sigma(p(u_{\theta},v_{\theta})) + \left\|p(u_{\theta},v_{\theta})-p(u_{\theta'} \oplus r, v_{\theta'} \oplus s)\right\|\cdot [-1,1].
\end{align*}
Combining Corollary \ref{cor:pol-diff} with Lemma \ref{lem:dil-approx}, we find a constant $\kappa_p$ depending only on $p$ with
\[
\left\|p(u_{\theta},v_{\theta})-p(u_{\theta'} \oplus r, v_{\theta'} \oplus s)\right\|\leq 2\kappa_p \sqrt{1-c^{2}} .
\]
From $1-e^{-x}\leq x$ for all $x\in\mathbb R$, we conclude
\[
\sqrt{1-c^{2}}= \sqrt {1-e^{-\frac{\left\vert\theta-\theta'\right\vert}{2}}}\leq \sqrt{\frac{\left\vert\theta-\theta'\right\vert}{2}}.
\]
Putting everything together we find that
\[
\sigma(p(u_{\theta'},v_{\theta'})) \subseteq \sigma(p(u_{\theta},v_{\theta})) + \sqrt{2} \kappa_p\left\vert\theta-\theta'\right\vert^{\frac{1}{2}} \cdot[-1,1].
\]
Since we have the same with $\theta$ and $\theta'$ interchanged, we get
\[
d\Bigl(\sigma\bigl(p(u_{\theta},v_{\theta})\bigr),\sigma\bigl(p(u_{\theta'},v_{\theta'})\bigr)\Bigr)\leq \sqrt{2} \kappa_p\left\vert\theta-\theta'\right\vert^{\frac{1}{2}} ,
\]
where $d$ denotes the Hausdorff-distance.
\end{proof}

\begin{corollary}\label{cor:Hcontfield}
  The family of rotation algebras $(A_\theta)_{\theta \in \bR}$ forms a continuous field of C*-algebras, in the sense that for every $*$-polynomial $p$ in two noncommuting variables, the function $\theta \mapsto \|p(u_\theta,v_\theta)\|$ is continuous.
More precisely, the function $\theta \mapsto \|p(u_\theta,v_\theta)\|$ is $\frac{1}{2}$-H{\"o}lder continuous.  
\end{corollary}
\begin{proof}
The corollary follows immediately from the theorem in the case that $p$ is a selfadjoint $*$-polynomial, because then $\|p(u_\theta, v_\theta)\| = \sup \{|\lambda| : \lambda \in \sigma(p(u_\theta,v_\theta))\}$.
But if $p$ is an arbitrary $*$-polynomial in two noncommuting variables, then $q = p^*p$ is a selfadjoint $*$-polynomial. 
This already implies that $\theta \mapsto \|p(u_\theta,v_\theta)\| = \|q(u_\theta,v_\theta)\|^{1/2}$ is continuous. 
If $c:=\|p(u_\theta,v_\theta)\| \neq 0$, then $\|p(u_\lambda,v_\lambda)\| > c/2 > 0$ for all $\lambda$ in a neighborhood of $\theta$. 
Thus, for all $\lambda, \mu$ in this neighborhood of $\theta$, 
\begin{align*}
\Bigl|\|p(u_\lambda,v_\lambda)\| - \|p(u_{\mu},v_{\mu})\|\Bigr|
&= \left|\frac{\|p(u_\lambda,v_\lambda)\|^2 - \|p(u_{\mu},v_{\mu})\|^2}{\|p(u_\lambda,v_\lambda)\| + \|p(u_{\mu},v_{\mu})\|}\right| \\
&\leq C\Bigl|\|q(u_\lambda,v_\lambda)\| - \|q(u_{\mu}, v_{\mu})\| \Bigr| \\
&\leq C'|\lambda - \mu|^{1/2}.
\end{align*}
This shows $\frac{1}{2}$-H{\"o}lder continuity in a neighborhood of $\theta$.

If $\|p(u_\theta,v_\theta)\| = 0$, we want to show that there is an $N\in\mathbb N$ and a $*$-polynomial $\widetilde p$ in two noncommuting variables such that $\widetilde{p}(u_\theta,v_\theta)\neq0$ and $p(u_{\lambda},v_{\lambda})=(e^{i\lambda}-e^{i\theta})^N\widetilde{p}(u_{\lambda},v_{\lambda})$ for all $\lambda\in\mathbb R$. 
As we already know that $\|\widetilde p(u_\lambda,v_\lambda)\|$ varies $\frac{1}{2}$-H{\"o}lder continuously with $\lambda$ in a neighborhood of $\theta$ and, obviously, $|e^{i\lambda}-e^{i\theta}|^N$ is even Lipschitz continuous in $\lambda$, we can then easily conclude $\frac{1}{2}$-H{\"o}lder continuity of $\lambda\mapsto \|p(u_\lambda,v_\lambda)\|=|e^{i\lambda}-e^{i\theta}|^N \|\widetilde{p}(u_\lambda,v_\lambda)\|$ in the given neighborhood of $\theta$. 

Consider the $*$-algebra defined by generators and relations 
\[\mathcal A:=\operatorname{*-alg}(U,V,Q \text{ unitary}\mid Q=VUV^*U^*, QU=UQ, QV=VQ)\]
and let $\pi_\lambda\colon \mathcal A\to \mathcal A_\lambda:=\operatorname{*-alg}(u_\lambda,v_\lambda)$ be the canonical homomorphism sending $U$ to $u_\lambda$, $V$ to $v_\lambda$ and $Q$ to $e^{i\lambda}$. 
To avoid confusion, let us stress that $\cA$ is an abstract algebra and that $U$, $V$ and $Q$ denote the generators of this algebra (not operators). 
The relations of $\mathcal A$ imply
\begin{align*}
  VU&=UVQ&V^*U&=UV^*Q^*\\
  VU^*&=U^*VQ^*& V^*U^*&=U^*V^*Q
\end{align*}
and, thus, they can be used to write the element $p(U,V)$ in the form
\[p(U,V)=\sum_{k,\ell\in\mathbb Z} U^k V^\ell f_{k,\ell}(Q)\]
with $f_{k,\ell}(z)$ appropriate (finitely many nonzero) Laurent polynomials in one variable $z$.
Accordingly, $p(u_\lambda,v_\lambda)=\pi_\lambda(p(U,V))=\sum u_\lambda^kv_\lambda^\ell f_{k,\ell}(e^{i\lambda})$. Now, $p(u_\theta,v_\theta)=0$ if and only if $f_{k,\ell}(e^{i\theta})=0$ for all $k,\ell$, because the monomials $u_\theta^{k}v_\theta^\ell$ form a basis of $\mathcal A_\theta$. If a Laurent polynomial $f(z)$ vanishes at $e^{i\theta}$, it can be factored as $(z-e^{i\theta}) \widetilde{f}(z)$ for some Laurent polynomial $\widetilde{f}(z)$. There is a maximal power $(z-e^{i\theta})^N$ which we can factor out from all $f_{k,\ell}(z)$ simultaneously, so that $f_{k,\ell}(z)=(z-e^{i\theta})^N \widetilde f_{k,\ell}(z)$ with some Laurent polynomials $\widetilde f_{k,\ell}(z)$, not all of which vanish at $z=e^{i\theta}$. We find 
\begin{align*}
  p(u_{\lambda},v_{\lambda})=\pi_{\lambda}(p(U,V))
  =\pi_{\lambda}\Bigl(\sum_{k,\ell\in\mathbb Z} U^k V^\ell f_{k,\ell}(Q)\Bigr)
  &=\pi_{\lambda}\Bigl((Q-e^{i\theta})^N\sum_{k,\ell\in\mathbb Z} U^k V^\ell \widetilde{f}_{k,\ell}(Q)\Bigr)\\
  &=(e^{i\lambda}-e^{i\theta})^N \widetilde{p}(u_{\lambda},v_{\lambda})
\end{align*}
for all $\lambda\in\mathbb R$ and some fixed $*$-polynomial $\widetilde p$ with $\widetilde p(u_\theta,v_\theta)\neq0$ as needed.

We have shown that $\lambda\mapsto \|p(u_\lambda,v_\lambda)\|$ is locally $\frac{1}{2}$-H{\"o}lder continuous. It follows that $\lambda\mapsto \|p(u_\lambda,v_\lambda)\|$ is $\frac{1}{2}$-H{\"o}lder continuous on the compact interval $[0,2\pi]$ and, thus, everywhere by periodicity.
\end{proof}

\begin{remark}
In the selfadjoint case, the $\frac{1}{2}$-H{\"o}lder continuity of $\|p(u_\theta,v_\theta)\|$ as a consequence of the $\frac{1}{2}$-H{\"o}lder continuity of $\sigma(p(u_\theta,v_\theta))$ is in accordance with \cite[Theorem 3]{BB16}, though we do not require that result here. 
\end{remark}



\section{A lower bound on the dilation scale}\label{sec:lower-bound-dilation}

\begin{proposition}\label{prop:lowerbnd}
Let $U,V$ be a $q$-commuting pair of unitaries.
If $M,N$ is a pair of commuting normals dilating $U,V$ such that $\|M\|,\|N\| \leq c$, then $c\geq \frac{4}{\|u+u^*+v+v^*\|}$.
\end{proposition}

\begin{proof}
First, assume that $q$ is a primitive $n$th root of unity.
Our goal is to show that for every $r <  \frac{4}{\|u+u^*+v+v^*\|}$, there exists a matrix valued polynomial $P$ of degree $1$, such that
\[
\|P(U,V)\| > \sup_{(z, w) \in r\overline{\bD}^2} \|P(z, w)\| = \sup_{(z, w) \in r\bT^2} \|P(z, w)\|.
\]
Clearly, this will show that no commuting normal dilation $M,N$ with norm $\|M\|,\|N\|\leq r$ exists for the pair $U,V$.
We may assume that $U,V$ are equal to a pair $\alpha X,\beta Y$ for $\alpha,\beta\in\mathbb T$ and $X,Y$ as in equation \eqref{eq:XY}; indeed, an arbitrary $(U,V)$ can be identified with the direct sum of its irreducible representations (which are all of that form by the remarks following Definition~\ref{def:rotation-algebra}) and the inequality holds if and only if it holds for one of the direct summands.

Consider the matrix valued polynomial $P(z,w)=P_\lambda(z,w)=\lambda I + z\overline\alpha X^* + w\overline{\beta}Y$, where $\lambda > 0$.
Then
\[
P(U,V) = \lambda I \otimes I + X^* \otimes X + Y \otimes Y.
\]
Clearly $\|P(U,V)\| \leq \lambda + 2$.
On the other hand,  let $\xi = \frac{1}{\sqrt{n}} \sum_{j=1}^n  e_j \otimes e_j$, where $e_1, \ldots, e_n$ is the standard orthonormal basis of $\bC^n$ (with respect to which the matrices in \eqref{eq:XY} are given).
By a direct computation, $\xi$ is a fixed vector for $X^* \otimes X$ and for $Y \otimes Y$, so $\|P(U,V) \xi\| = \|(\lambda + 2)\xi\| = \lambda + 2$.
We conclude that $\|P(U,V)\| = \lambda + 2$.

Now let us evaluate $\sup_{(z, w) \in r\bT^2} \|P(z, w)\|$.
\begin{align*}
\sup_{(z, w) \in r\bT^2} \|P(z, w)\|^2
  & = \sup_{(\alpha', \beta') \in \bT^2}  \|\lambda + r (\alpha'\overline{\alpha} X^* + \beta'\overline{\beta} Y)\|^2 \\
  &\leq \sup_{(\gamma, \delta) \in \bT^2} \lambda^2 + \lambda r \|\gamma X^* + \overline{\gamma} X + \delta Y + \overline{\delta} Y^*\| + 4r^2 \\
  & = \lambda^2 + \lambda r \|u_\theta+u_\theta^*+v_\theta+v_\theta^*\| + 4r^2.
\end{align*}
where the last equality holds because the norm of $u_\theta + u_\theta^* + v_\theta + v_\theta^*$ is given by the supremum of the norm of its image over all irreducible representations. 
We will consider the expression
\be\label{eq:quotient_lambda}
\frac{1}{\lambda} \left(\|P(U,V)\|^2-\sup_{(z, w) \in r\bT^2}\|P(z,w)\|^2 \right) ,
\ee
and examine its value for large values of $\lambda$.
If we show that this becomes positive whenever $r <  \frac{4}{\|u+u^*+v+v^*\|}$, then the proof would be complete.
Indeed, plugging in the above results, we see that
\begin{align*}
  \frac{1}{\lambda} \left(\|P(U,V)\|^2-\sup_{(z, w) \in r\bT^2}\|P(z,w)\|^2 \right)
  \geq 4-r \|u_\theta+u_\theta^*+v_\theta+v_\theta^*\|+\frac{4-4r^2}{\lambda}
\end{align*}
Since the right hand side expression is positive for large enough $\lambda$ when $r < \frac{4}{\|u+u^*+v+v^*\|}$, the proof is complete in the case that $q$ is a root of unity.

Finally, if $q = e^{i\theta}$ is not a root of unity, then the lower bound also holds, since both the dilation constant $c_\theta$ as well as the norm $\|u_\theta + u_\theta^* + v_\theta + v_\theta^*\|$ depend continuously on $\theta$.
\end{proof}

\begin{remark}\label{rem:irrational}
  In the proof, the assumption that $q$ is a primitive $n$th root of unity is mainly used to show that $\|P(U,V)\| = \lambda + 2$. For an irrational rotation this can also be established directly (with $P$ now an appropriate operator valued polynomial), without relying on the continuity argument. Indeed, for $\frac{\theta}{2\pi}$ irrational, it is a simple consequence of the fact that $C^*(U,V)\cong A_\theta$ and Lemma \ref{RR} that the canonical homomorphism $C(\mathbb{T}^2)\cong A_0\to C^*(U^* \otimes U , V \otimes V)$ is an isomorphism, so $\|P(U,V)\| = \|\lambda + u_0 + v_0\|=\lambda + 2$.
\end{remark}

\section{An optimal dilation}\label{sec:an-optimal-dilation}

\begin{theorem}\label{thm:optdil_qqtag}
Let $\theta,\theta' \in \bR$, set $q = e^{i\theta}$ and $q' = e^{i\theta'}$, and put $\gamma = \theta'-\theta$.
Consider two $q$-commuting unitaries $U,V$.
Then $U,V$ can be dilated to a to a pair of $q'$-commuting scalar multiples of unitaries with norm $\frac{4}{\|u_\gamma+u_\gamma^*+v_\gamma+v_\gamma^*\|}$. 
Furthermore, if $\frac{q'}{q}=e^{i\gamma}$ is a primitive $n$th root of unity and $U,V\in B(\cH)$, then such a dilation can be constructed on $\cH\otimes \cL$ with $\dim(\cL)= n$. 
\end{theorem}

\begin{proof}
Represent $C^*(U,V)$ concretely on a Hilbert space $\cH$.
Let $u_\gamma, v_\gamma$ be the universal generators of $A_\gamma$ and put $h_\gamma:=u_\gamma+u_\gamma^*+v_\gamma+v_\gamma^*$.
We claim that there exists a state $\varphi$ on $A_\gamma$ such that $|\varphi(u_\gamma)| = |\varphi(v_\gamma)| = \frac{\|h_\gamma\|}{4}$.
Assuming the existence of such a state for the moment being, let us define
\[
U' = U \otimes \frac{\pi(u_\gamma)}{\varphi(u_\gamma)} \quad , \quad V' = V \otimes \frac{\pi(v_\gamma)}{\varphi(v_\gamma)}.
\]
on $\cK = \cH \otimes \cL$, where $\pi\colon A_\gamma\to B(\cL)$ is the GNS representation of $\varphi$. 
These are $q'$-commuting scalar multiples of unitaries, and they have norm $\frac{4}{\|h_\gamma\|}$. 
By construction, there exists a unit vector $x \in \cL$ such that $\varphi(a) = \langle  x ,\pi(a) x \rangle$ for all $a \in A_\gamma$.
Consider the isometry $W \colon \cH \to \cH \otimes \cL$ defined by
\[
W h = h \otimes x \quad , \quad h \in \cH.
\]
Then $W^* U' W = \frac{1}{\varphi(u_\gamma)} \langle  x,\pi(u_\gamma) x\rangle U = U$ and $W^*V' W = \frac{1}{\varphi(v_\gamma)} \langle x,\pi(v_\gamma) x\rangle V = V$, and the proof of the existence of a dilation is complete.

To show the existence of a state $\varphi$ as above, we first choose a state $\psi$ such that
\[
|\psi(u_\gamma+u_\gamma^*+v_\gamma+v_\gamma^*)| = \|u_\gamma+u_\gamma^*+v_\gamma+v_\gamma^*\| .
\]
As $u_\gamma v^*_\gamma = e^{i\gamma} v^*_\gamma u_\gamma$ and $A_\gamma$ is universal, there exists an automorphism $\sigma \in \Aut(A_\gamma)$ determined by $\sigma(u_\gamma) = v_\gamma^*$ and $\sigma(v_\gamma) = u_\gamma$.
Then $\varphi := \frac{1}{4}(\psi + \psi \circ \sigma + \psi \circ \sigma^2 + \psi \circ \sigma^3)$ is a state that satisfies
\[
\varphi(u_\gamma) = \varphi(v_\gamma) = \frac{\psi(u_\gamma+u_\gamma^*+v_\gamma+v_\gamma^*)}{4} ,
\]
so $|\varphi(u_\gamma)| = |\varphi(v_\gamma)| = \frac{\|h_\gamma\|}{4}$ as required.

Now assume that $e^{i\gamma}=\frac{q'}{q}$ is a primitive $n$th root of unity. 
All we have to do, is to find a state $\varphi$ as above, so that the GNS representation space $\cL$ is $n$ dimensional. 
By Lemma~\ref{lem:pi_11} below, the matrices $X=\pi_{1,1}(u_\gamma), Y=\pi_{1,1}(v_\gamma)$ given by \eqref{eq:XY} fulfill 
\[
\|u_\gamma+u_\gamma^*+v_\gamma+v_\gamma^*\| = \|X+X^*+Y+Y^*\| ,
\]
(see also \cite[Corollary~4.8]{BocBook}). 
Therefore, there is a unit vector $x\in\mathbb{C}^n$ with 
\[
\|h_\gamma\|=\|X+X^*+Y+Y^*\|=|\langle x,(X+X^*+Y+Y^*) x\rangle|. 
\] 
Resembling the automorphism $\sigma$ above, there is an automorphism of $C^*(X,Y)=M_n(\mathbb{C})$ with $X\mapsto Y^*$ and $Y\mapsto X$, which is implemented by the unitary matrix 
\begin{align}
\mathcal F=\left(\frac{1}{\sqrt{n}}e^{i\gamma k\ell}\right)_{k,\ell=0}^{n-1}\in M_n(\mathbb{C})\label{eq:dFT}
\end{align}
i.e., $\mathcal{F}^*X\mathcal{F}=Y^*$ and $\mathcal{F}^*Y\mathcal{F}=X$ (this can be thought of as a version of the discrete Fourier transform on $\mathbb{C}^n$). 
In particular, $\mathcal{F}$ commutes with $X+X^*+Y+Y^*$, so we can choose $x$ to be an eigenvector of $\mathcal F$ as well. 
(In fact, $x$ is automatically an eigenvector of $\mathcal F$, because all nonzero eigenvalues of $X+X^*+Y+Y^*$ are simple; see the proof of Theorem 3.3 in \cite{CEY}.) 
Therefore, even without an extra symmetrization argument, we find
\[
\langle x,Xx\rangle=\langle x,\mathcal{F}^*Y\mathcal{F}x\rangle =\langle \mathcal{F}x,Y\mathcal{F}x\rangle = \langle x,Yx\rangle.
\]
We conclude that $\varphi:=\langle x,\pi_{1,1}(\cdot) x \rangle$ is a state on $A_\gamma$ with GNS-representation $\pi_{1,1}$ on $\mathbb{C}^n=:\mathcal{L}$ and $|\varphi(u_\gamma)| = |\varphi(v_\gamma)| = \frac{\|X+X^*+Y+Y^*\|}{4} = \frac{\|h_\gamma\|}{4}$, as required. 
\end{proof}

\begin{corollary}\label{thm:optdil}
Consider the generators $u_\theta,v_\theta$ of the rotation C*-algebra $A_\theta$.
Then $u_\theta,v_\theta$ can be dilated to a to a pair of commuting normals with norm $\frac{4}{\|u_\theta+u_\theta^*+v_\theta+v_\theta^*\|}$.
\end{corollary}

Putting together Proposition \ref{prop:equiv}, Proposition \ref{prop:lowerbnd} and Corollary \ref{thm:optdil}, we obtain the following exact expression of the dilation constant.

\begin{theorem}\label{thm:opt_dilconst}
For every $\theta \in \bR$,
\[
c_\theta = \frac{4}{\|u_\theta+u_\theta^*+v_\theta+v_\theta^*\|}.
\]
In fact, for every pair of $e^{i\theta}$-commuting unitaries $U,V$, there exists a pair $U_0,V_0$ of commuting unitaries, such that $c_\theta U_0,c_\theta V_0$ is a dilation of $U,V$.
\end{theorem}

The dilation found in Theorem \ref{thm:optdil_qqtag} is also optimal in the general situation
, where $q$-commuting unitaries are dilated to $q'$-commuting unitaries with $\theta,\theta' \in \bR$, $q = e^{i\theta}$, $q' = e^{i\theta'}$, and $\gamma = \theta'-\theta$.
Indeed, consider the operator valued polynomial $P_\lambda(z,w)=\lambda + zu_{-\theta} + w v_{-\theta}$, where $\lambda >0$.
Just as in Remark \ref{rem:irrational}, we can conclude that $\|P_\lambda(u_\theta,v_
\theta)\|=\lambda + 2$.
On the other hand, we can identify $u_{\theta'} \otimes u_{-\theta}, v_{\theta'} \otimes v_{-\theta}$ with $u_\gamma, v_\gamma$.
Thus, writing $h_\gamma = u_\gamma + u_\gamma^* + v_\gamma + v_\gamma^*$,
\begin{align*}
\|P_\lambda(ru_{\theta'},rv_{\theta'})\|^2
& = \|\lambda + ru_\gamma + rv_\gamma\|^2 \\
& = \|(\lambda + ru_\gamma + rv_\gamma)^*(\lambda + ru_\gamma + rv_\gamma)\| \\
& = \|\lambda^2 + \lambda r h_\gamma + r^2(u_\gamma + v_\gamma)^*(u_\gamma + v_\gamma)\| \\
& \leq \lambda^2 + \lambda r \|h_\gamma\| + r^2\|u_\gamma + v_\gamma\|^2.
\end{align*}
So, similarly to the calculation in the proof of Proposition \ref{prop:lowerbnd}, we have that
\[
\limsup_{\lambda\to\infty}\frac{1}{\lambda} \left(\|P_\lambda(u_{\theta},v_{\theta})\|^2 - \|P_\lambda(ru_{\theta'},rv_{\theta'})\|^2 \right) \geq 4-r\|u_\gamma + u_\gamma^* + v_\gamma + v_\gamma^*\|.
\]
So if $ru_{\theta'},rv_{\theta'}$ is a dilation of $u_\theta,v_\theta$, it follows that $r\geq \frac{4}{\|u_\gamma + u_\gamma^* + v_\gamma + v_\gamma^*\|}$.
We record this result. 
\begin{theorem}\label{thm:optdil_gen}
Let $\theta,\theta' \in \bR$, $q = e^{i\theta}$, $q' = e^{i\theta'}$, and $\gamma = \theta'-\theta$. 
The smallest constant $c_{\theta,\theta'}$ such that every pair of $q$-commuting unitaries can be dilated to $c_{\theta,\theta'}$ times a pair of $q'$-commuting unitaries is given by 
\[
c_{\theta,\theta'} = c_\gamma = \frac{4}{\|u_\gamma + u_\gamma^* + v_\gamma + v_\gamma^*\|} .
\]
\end{theorem}

\section{Numerical value of the dilation constants}\label{sec:numericalval}\label{sec:numerical}

In order to calculate the numerical value of the optimal dilation constants $c_\theta$, we need to evaluate the norm of the operator $h_\theta = u_\theta + u_\theta^* + v_\theta + v^*_\theta$.
There are some very delicate analytical estimates for $\|h_\theta\|$; see \cite{BZ05}.
On the other hand, one look at Hofstadter's butterfly \cite[Fig.~1]{Hof76} (see also Figure \ref{fig:plot} below) is enough to convince that $\|h_\theta\|$ is a non-smooth function of $\theta$, and thus we do not expect to be able to give a closed form analytic expression of $c_\theta$.
However, we do expect to be able to numerically approximate $c_\theta$ to very high precision.
First, the Lipschitz continuity of $c_\theta$ in the parameter $\theta$ allows us to focus on finding the value $c_\theta$ for rational angles.
Second, it turns out that we do not have to deal with the universal operators $u_\theta, v_\theta$, but can restrict attention to certain finite dimensional representations.
This reduction is achieved in the following lemma (the lemma is known --- see, e.g., \cite[Corollary 4.8]{BocBook} --- but we include a direct proof for completeness).

\begin{lemma}\label{lem:pi_11}
Let $A_\theta = C^*(u_\theta,v_\theta)$ be a rational rotation $C^*$-algebra and denote by $X = \pi_{1,1}(u_\theta)$ and $Y = \pi_{1,1}(v_\theta)$ the images of $u_\theta$ and $v_\theta$ under the standard representation as in Equation \eqref{eq:XY}.
Then
\[
\|u_\theta+u_\theta^*+v_\theta+v_\theta^*\|=\|X + X^* + Y + Y^*\| .
\]
Furthermore, the right hand side norm is the largest positive eigenvalue of $X + X^* + Y + Y^*$ and this eigenvalue is attained at an eigenvector with nonnegative entries. 
\end{lemma}

\begin{proof}
   We know that  $\|u_\theta+u_\theta^*+v_\theta+v_\theta^*\|= \sup_{(\alpha, \beta) \in \bT^2}\|\pi_{\alpha,\beta}(u_\theta+u_\theta^*+v_\theta+v_\theta^*)\|$, so we have to show that $\|\pi_{\alpha,\beta}(u_\theta+u_\theta^*+v_\theta+v_\theta^*)\|\leq \|\pi_{1,1}(u_\theta+u_\theta^*+v_\theta+v_\theta^*)\|$ for all $\alpha,\beta\in\mathbb T$. In the following we will write $h^{(\alpha,\beta)}:=\pi_{\alpha,\beta}(u_\theta+u_\theta^*+v_\theta+v_\theta^*)=\alpha X + \overline\alpha X^* + \beta Y + \overline\beta Y^*$. First, note that $h^{(-\alpha,-\beta)}=-h^{(\alpha,\beta)}$. So for at least one of the two selfadjoint matrices $h^{(\alpha,\beta)}$ and $h^{(-\alpha,-\beta)}$ its norm is given by the largest eigenvalue. Let $\alpha=e^{i\gamma}$. For some unit vector $x\in \mathbb C^n$ and a sign $\varepsilon\in\{+1,-1\}$, we have (with the convention that $x_{n+1}=x_1$)
  \begin{align*}
    \|h^{(\alpha,\beta)}\|
    =\|\varepsilon h^{(\alpha,\beta)}\|
    &=\langle x,\varepsilon h^{(\alpha,\beta)}x\rangle\\
    &=\varepsilon\left(\sum_{k=1}^{n} 2\cos (k\theta + \gamma) |x_k|^2 + 2\re \beta\overline{x_k}x_{k+1}\right)\\
    &\leq    \sum_{k=1}^{n} 2\varepsilon\cos (k\theta + \gamma) |x_k|^2 + 2 \left|x_k\right|\left|x_{k+1}\right|\\
    &=\langle \left|x\right|,\pi_{\varepsilon\alpha,1}(u_\theta+u_\theta^*+v_\theta+v_\theta^*)\left|x\right|\rangle\\
    &\leq \|h^{(\varepsilon\alpha,1)}\|
  \end{align*}
  where $\left|x\right|:=(|x_k|)_{k=1}^n$.
  Since $h^{(\alpha,\beta)}$ and $ h^{(\overline\beta,\alpha)}$ are unitarily equivalent by the modified discrete Fourier transformation \eqref{eq:dFT}, we get $\|h^{(\varepsilon\alpha,1)}\| = \|h^{(1,\varepsilon\alpha)}\|$ and can repeat the same argument for $\|h^{(1,\varepsilon\alpha)}\|$ to conclude that $\|h^{(1,\varepsilon\alpha)}\|\leq \|h^{(\delta,1)}\|=\|h^{(1,\delta)}\|$ for some sign $\delta\in\{+1,-1\}$. If $\delta=1$, we are done. If $\delta=-1$ note that $h^{(1,-1)}$ is unitarily equivalent to $h^{(-1,1)}=-h^{(1,-1)}$, in particular, its norm is given by the largest eigenvalue, because the spectrum is symmetric. So we can apply the above procedure a third time with $\varepsilon=1,\alpha=1,\beta=-1$ to show $\|h^{(1,-1)}\|\leq\|h^{(1,1)}\|$, which finishes the proof of the first statement.

  If we now go through the very same calculation for $h^{(1,1)}$, we find that
  \[\|h^{(1,1)}\|\leq \langle |x|,h^{(\varepsilon,1)}|x|\rangle \leq \|h^{(\varepsilon,1)}\|\]
  for some unit vector $x$ and some sign $\varepsilon$. If $\varepsilon=1$, we see that $|x|$ must be a positive eigenvector with eigenvalue $\|h^{(1,1)}\|$. If $\varepsilon=-1$, a final round for $h^{(1,-1)}$ yields
  \[\|h^{(1,1)}\|\leq \|h^{(-1,1)}\|=\|h^{(1,-1)}\|\leq \langle |x|,h^{(1,1)}|x|\rangle \leq \|h^{(1,1)}\|\]
  (as above, we can choose the sign to be $1$ here because $h^{(1,-1)}$ has symmetric spectrum).
  %
\end{proof}

In light of the above lemma, we numerically computed $\|u_\theta + u_\theta^* + v_\theta + v_\theta^*\|$ for rational angles.
In Figure \ref{fig:plot} we plotted a graph of $c_\theta$ as a function of $\theta$.
Examining the computed values shows that $\max_\theta c_\theta \geq 1.543$.

It is interesting to note that a minor adaptation of the Inclusion Algorithm from \cite[Section 4.1]{HKM13} can be used to directly compute the dilation constants for the standard representation of $A_\theta$ for rational $\theta$ (it is not hard to show that, in the rational case, the dilation constant for the standard representation is equal to the dilation constant for the universal representation, i.e., to $c_\theta$).
The results agree nicely with the values we obtain for $c_\theta$, as they should.
The Python code that was used for all of the above numerical computations can be found in the link \url{https://colab.research.google.com/drive/1imIjguPLWA6ll5mLU1DQ9bFqTT-t6jTe}.

\begin{figure}
\caption{The dilation constant $c_\theta$ as a function of $\theta$}
\includegraphics[scale=0.65]{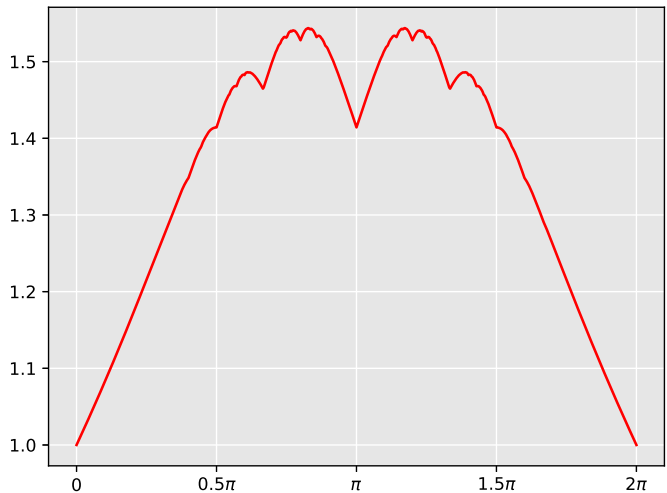}
\label{fig:plot}
\end{figure}

We are able to compute some values of $c_\theta$ explicitly, as in the following examples.

\begin{example}\label{ex:23}
  For $\theta=\frac{2}{3}\pi$, one can easily find by diagonalization that
  \[\|h_{\frac{2}{3}\pi}\|=\left\|
      \begin{pmatrix}
        2&1&1\\1&-1&1\\1&1&-1
      \end{pmatrix}
      \right\|
      =1+\sqrt{3}
    ,\]
  so $c_{\frac{2}{3}\pi}=\frac{4}{\|h_{\frac{2}{3}\pi}\|}=2\sqrt{3}-2\approx 1.4641$, showing that the lower bound in \eqref{eq:Cd} can be improved.
\end{example}
\begin{example}\label{ex:45}
  Consider $\theta=\frac{4}{5}\pi$. Then we have $2\cos\theta=2\cos 4\theta=-\gamma$ and $2\cos 2\theta=2\cos 3\theta= \gamma-1$, where $\gamma=\frac{1+\sqrt{5}}{2}$ is the golden ratio, so
  \begin{align*}
    h:=X+X^*+Y+Y^*=
    \begin{pmatrix}
      2&1&0&0&1\\1&-\gamma&1&0&0\\0&1&\gamma-1&1&0\\0&0&1&\gamma-1&1\\1&0&0&1&-\gamma
    \end{pmatrix}
  \end{align*}
  which has the positive eigenvector $x=(2\gamma,1,1,1,1)$ with eigenvalue $\gamma + 1$. Note that an eigenvector for another eigenvalue must be orthogonal to $x$ and, thus, cannot have nonnegative entries. So, Lemma~\ref{lem:pi_11}  yields $\|h\|=\gamma + 1$. 
This gives $c_{\frac{4}{5}\pi} = 4/(\gamma+1) \approx 1.5279$.
However, as we saw, this is still not the highest value of $c_\theta$.
In fact, the maximum of $c_\theta$ cannot be attained at any rational angle (see Equation 4.2 in \cite{BR90}).
\end{example}

We close with conjecturing that the maximum of $c_\theta$ is actually attained at the so-called \emph{silver mean} $\theta_s=\frac{2\pi}{\gamma_s}=2\pi(\sqrt{2}-1)$, where $\gamma_s=\sqrt{2}+1$ is the silver ratio. This is strongly supported by our numerical tests. 
The silver mean is also an important special parameter in the study of self-similarity of the Hofstadter butterfly; see, e.g., \cite{RP97}. 
Using the rational approximation $\frac{1}{\gamma_s}\approx \frac{2378}{5741}$ with an error of less than $1.1\cdot 10^{-8}$ and the Lipschitz continuity of $c_\theta$ stated in Corollary~\ref{cor:ctheta}, we find that $c_{\theta_s}\approx 1.5437772$ with an error of less than $10^{-7}$. 
This is done by computing the norm of the $5741 \times 5741$ selfadjoint matrix $h = X + X^* + Y + Y^*$, where $X,Y$ are the standard representation of $u_\theta, v_\theta$ with $\theta = \frac{2378}{5741}2\pi$. 
We believe that our error bound of $10^{-7}$ is reliable, because we can represent $h$ up to two times the machine precision ($\approx 2\times 10^{-16}$) in operator norm (the error is due to the inexact values $2 \cos(k \theta)$ appearing on the diagonal of the computer representation of $h$), and then the computation of the norm of a selfadjoint matrix is a stable numerical task.

\section*{acknowledgments}

The authors would like to thank Joseph Avron, Siegfried Beckus, Florin Boca, Yoram Last and Terry Loring for some helpful discussions and for suggesting references and leads.
Benjamin Passer and Ron Rosenthal both contributed insightful comments at key moments.
Special thanks go to Mattya Ben-Efraim and Yuval Yifrach for running interesting preliminary numerical computations which suggested that an examination of the dilation constants for $q$-commuting unitaries is called for. Finally, this paper was significantly improved following the remarks of four anonymous referees, to whom we are grateful. 


\bibliographystyle{amsplain}

\end{document}